\documentclass{amsart}

\usepackage{amssymb}
\usepackage[dvips]{epsfig}
\usepackage{graphicx}
\usepackage{color}

 \usepackage[latin1]{inputenc}
 \usepackage[T1]{fontenc}
 \usepackage[normalem]{ulem}
\usepackage{verbatim}
 \usepackage{graphicx}
\usepackage[latin1]{inputenc}
\usepackage{amsmath}
\usepackage{amssymb}
\usepackage{amsfonts}
\usepackage{amsthm}
\usepackage{bbm}

\newcommand{\N}{\mathbb{N}}
\newcommand{\R}{\mathbb{R}}

\newtheorem{theorem}{Theorem}[section]
\newtheorem{lemma}[theorem]{Lemma}
\newtheorem{proposition}[theorem]{Proposition}
\newtheorem{corollary}[theorem]{Corollary}
\theoremstyle{remark}
\newtheorem{remark}{Remark}[section]
\theoremstyle{definition}
\newtheorem{definition}[theorem]{Definition}

\newtheorem*{merci}{Acknowledgements}
\numberwithin{equation}{section}
\begin{document}
\title[Orbital stability of ground state of fKdV]{Remarks on the orbital stability of ground state solutions of fKdV and related equations }
\author[F. Linares]{Felipe Linares}
\address{ IMPA\\ Estrada Dona Castorina 110\\ Rio de Janeiro 22460-320, RJ Brasil}
\email{ linares@impa.br}
\author[D. Pilod]{Didier Pilod}
\address{Instituto de Matem\' atica, Universidade Federal do Rio de Janeiro, Caixa Postal 68530 CEP 21941-97, Rio de Janeiro, RJ Brasil}
\email{didier@im.ufrj.br}
\author[J.-C. Saut]{Jean-Claude Saut}
\address{Laboratoire de Math\' ematiques, UMR 8628,\\
Universit\' e Paris-Sud et CNRS,\\ 91405 Orsay, France}
\email{jean-claude.saut@math.u-psud.fr}


\maketitle

\begin{abstract}
The aim of this paper is to provide a proof of the (conditional) orbital stability of solitary waves solutions to the fractional Korteweg- de Vries equation (fKdV) and to the fractional Benjamin-Bona-Mahony (fBBM) equation in the $L^2$ subcritical case. We also discuss instability and its possible scenarios.
\end{abstract}

\large
\section{Introduction}
This paper continues the study initiated in \cite{LPS2} of the fractional Korteweg-de Vries equation (fKdV)

\begin{equation}\label{fKdV}
 u_t+uu_x-D^\alpha u_x=0,\quad u(\cdot,0)=u_0 \, ,
\end{equation}
and of its Benjamin-Bona-Mahony counterpart (fBBM)
\begin{equation}\label{fBBM}
 u_t+u_x+uu_x+D^\alpha u_t=0 \, ,
\end{equation}
where  $D^\alpha=(-\partial_x^2)^{\frac{\alpha}2}$ and $0<\alpha <1$. $D^{\alpha}$ is defined via Fourier transform by
\begin{displaymath}
\big(D^{\alpha}f\big)^{\wedge}(\xi)=|\xi|^{\alpha}\widehat{f}(\xi) \, .
\end{displaymath}
  
The fKdV equation is a toy model to understand the interaction between nonlinearity and dispersion. The choice is here to fix the quadratic nonlinearity which appears \lq\lq generically\rq\rq \, in many physical contexts and to vary (lower) the dispersion (see \cite{LPS2, KS}). 
     
     Equations like \eqref{fKdV} but with an inhomogeneous symbol can be derived rigorously as water waves models (in the small amplitude, long wave regime) \cite {La,LS}. For instance the so-called Whitham equation \cite{W} is of fKdV type with a weak dispersion, that is
\begin{equation}\label{Whit}
u_t+uu_x+\int_{-\infty}^{\infty}k(x-y)u_x(y,t)dy=0 \, .
\end{equation}
This equation can also be written on the form
\begin{equation}\label{Whibis}
u_t+uu_x-Lu_x=0 \, ,
\end{equation}
where the Fourier multiplier operator $L$ is defined by 
$$\widehat{Lf}(\xi)=p(\xi)\hat{f}(\xi) \, ,$$
with $p=\hat{k}.$
In the original Whitham equation, the kernel $k$ was given by 
\begin{equation}\label{tanh}
k(x)=\frac{1}{2\pi}\int_\R \left( \frac{\tanh \xi}{\xi} \right)^{\frac12} e^{ix\xi} d\xi,
\end{equation}
that is $p(\xi)=\left( \frac{\tanh \xi}{\xi} \right)^{\frac12} $ which behaves like $|\xi|^{-\frac12}$ for large frequencies and like $1-\frac{\xi^2}{6}$ for small frequencies. 

When surface tension is included  the symbol $p$ above has to be changed to $p_S(\xi)=(1+\beta|\xi|^2)^{\frac12}\left( \frac{\tanh \xi}{\xi} \right)^{\frac12},$ where $\beta\geq 0$ measures the surface tension effects. This leads to the {\it extended Whitham equation} where the symbol $p_S(\xi)$ behaves as $\beta|\xi|^{\frac12}$ for large frequencies and as $1-(\frac{1}{6}-\beta)\xi^2$ for small frequencies.

     \vspace{0.3cm}
  The equation  \eqref{fKdV} is invariant under the scaling transformation 
\begin{displaymath} 
u_{\lambda} (x,t)=\lambda^{\alpha}u(\lambda x,\lambda^{\alpha+1}t),
\end{displaymath}
for any positive number $\lambda$. A straightforward computation shows that $\|u_{\lambda}\|_{\dot{H}^s}$ $=\lambda^{s+\alpha-\frac{1}{2}}\|u\|_{\dot{H}^s}$, in particular the value $\alpha =\frac{1}{2}$ corresponds to the $L^2$ critical case.

One associates to  \eqref {fKdV}, \eqref{fBBM}  the energy space $H^{\frac{\alpha}2}(\R), $ motivated by their conservation laws. The following quantities are formally conserved by the flow associated to \eqref{fKdV},
\begin{equation} \label{M}
M(u)=\frac{1}{2}\int_{\mathbb R}u^2(x,t)dx,
\end{equation}
and 
\begin{equation} \label{H}
E(u)=\int_{\mathbb R}\big( \frac{1}{2} |D^{\frac{\alpha}2}u(x,t)|^2-\frac{1}{6}u^3(x,t)\big) dx.
\end{equation}
Note that by the Sobolev embedding $H^{\frac{1}{6}}(\R)\hookrightarrow L^3(\R)$,  $H(u)$ is well-defined if and only if $\alpha \geq \frac{1}{3}, $ in other words $\alpha=\frac{1}{3}$ is the energy critical exponent.

On the other hand, there is no energy critical exponent $\alpha$ in the case of the fBBM equation \eqref{fBBM} since the momentum $$N(u)=\frac{1}{2}\int_{\R}(u^2+|D^{\frac{\alpha}{2}}u|^2) dx$$
makes obviously always sense for $u\in H^{\frac{\alpha}2}(\R)$. Another conserved quantity for \eqref{fBBM} is the Hamiltonian $$F(u)=\int_\R \left(\frac{u^2}{2}+\frac{u^3}{6}\right) \, ,$$
which makes sense when $u\in H^{\frac{\alpha}2}(\R), \alpha \geq \frac{1}{3}.$

 \vspace{0.3cm}
 There is apparently no published result on the orbital stability for solitary waves  of  fractional KdV equations (fKdV) \eqref{fKdV} or fractional BBM equations  \eqref {fBBM} in the range $0<\alpha<1.$ The known existence  proofs (see \cite{FL, FLS} and also  \cite{ABS})   use M.  Weinstein's argument, looking for the best constant in the fractional Gagliardo-Nirenberg inequality
 
 \begin{equation}\label{GN}
 \int_\R|u|^3dx\leq C\left(\int_\R |D^{\frac{\alpha}2} u|^2 dx\right)^{\frac1{2\alpha}}\left(\int_\R u^2 dx\right)^{\frac{3\alpha -1}{2\alpha}} \, .
 \end{equation}
This gives the existence in the energy sub-critical case $\alpha>\frac{1}{3}$, but of course not any kind of stability, which should be true only in the $L^2$ subcritical case, $\alpha>\frac{1}{2}.$
 
 Orbital stability issues for the  fractional Schr\"{o}dinger equations has been considered in \cite{CHHO}.
We will restrict to the  fKdV equation \eqref{fKdV} with {\it homogeneous } dispersion, (but the method extends obviously to the non homogeneous case).

\smallskip    
    The solitary waves are solutions of \eqref{fKdV} of the form $u(x,t)=Q_c(x-ct), c>0$ where $Q_c$ belongs to the energy space $H^{\frac{\alpha}2}(\R) $ and they should thus satisfy the equation
 \begin{equation}\label{sol}
D^{\alpha} Q_c+cQ_c-\frac{1}{2}Q_c^2=0 \, .
 \end{equation}
The energy identity
\begin{equation} \label{EnergyIdentity}
\int_\R |D^{\frac{\alpha}2}Q_c|^2dx+c\int_\R Q_c^2 dx-\frac{1}{2}\int_\R Q_c^3 dx=0 
\end{equation}
and the Pohojaev identity
\begin{equation} \label{PohojaevIdentity}
\frac{\alpha-1}{2}\int_\R |D^{\frac{\alpha}2}Q_c|^2dx-\frac{c}{2}\int_\R Q_c^2 dx+\frac{1}{6}\int_\R Q_c^3 dx=0
\end{equation}
which in turn is a consequence of the identity (see for instance Lemma 3 in \cite{KMR})
\begin{equation} \label{Poh}
\int_{\mathbb R}(D^{\alpha}\phi)x\phi'dx=\frac{\alpha-1}2\int_{\mathbb R}|D^{\frac{\alpha}2}\phi|^2dx,
\end{equation}
imply
\begin{equation}\label{nonex}
(3\alpha-1)\int_\R |D^{\frac\alpha2} Q_c|^2 dx-c\int_\R Q_c^2dx=0
\end{equation}
proving that no finite energy solitary waves exist in the energy subcritical case $\alpha>1/3$ when $c\leq 0$ (see \cite{LPS2}).

 \vspace{0.3cm}   
     J. Albert has considered in \cite{A} the case $\alpha\geq 1,$ for \eqref{fKdV} so we will focus on the case $1/2<\alpha<1,$ which is $L^2$ sub-critical for \eqref{fKdV}. In his notation, $s= \alpha/2.$ The proof in \cite{A} is inspired by an old idea of Boussinesq, revisited by Benjamin \cite{B} (and by Cazenave-Lions \cite{CL}  for NLS type equations) and consists in using the concentration-compactness method of P.-L. Lions to prove the existence of a minimizer of the Hamiltonian (energy) with fixed momentum ($L^2$ norm). The proof gives nearly for free the orbital stability of the set of minimizers, assuming that the corresponding Cauchy problem is globally well-posed in the energy space, at least for initial data close to a solitary wave (a fact which is conjectured  but still unproved in the case of fKdV when $1/2<\alpha <1.)$
    
    Uniqueness and positivity properties of a class of solitary waves (ground states) have been investigated in \cite{FQT, F, FL,FLS} among others. 
    We recall that existence of solitary waves of arbitrary positive velocities has been established in the {\it energy subcritical case}, that is when $\alpha >\frac{1}{3}$ while no localized solitary waves exist when $0<\alpha <\frac{1}{3}$ (see the argument above), that is in the energy supercritical case. It is worth noticing that existence of solitary waves for the original Whitham equation has been established in \cite{EGW} by exploiting that the dispersion approaches that of the KdV equation in the long wave limit.

On the other hand, numerical simulations (\cite{KS}) suggest that the Cauchy problem for \eqref{fKdV} is globally well-posed for $\alpha >\frac{1}{2},$ a typical solution decomposing into solitary waves plus radiation, which would give a positive answer to the {\it soliton resolution conjecture} (\cite{Tao}). One aim of this note is to provide a (small) step towards this conjecture, namely to prove that the solitary waves are orbitally stable for this range of $\alpha$'s.\footnote{Actually we prove a {\it conditional} stability result since we do not know that the solutions of the Cauchy problem are global in this case.}

    \vspace{0.3cm}
    
The paper is organized as follows. In the following section we consider the fKdV equation. The next section deals with the fBBM equation. Lastly we initiate an extension to fractional Kadomtsev-Petviashvili I (fKPI) equations.

\vspace{0.3cm}
\noindent{\bf Notations.} We will denote $|\cdot|_p$ the norm in the Lebesgue space $L^p(\R),\; 1\leq p\leq \infty$ and $\|\cdot\|_s$ the norm in the Sobolev space $H^s(\R),\; s\in \R.$ We will denote $\hat {f}$ or $\mathcal F(f)$ the Fourier transform of a tempered distribution $f.$ For any $s\in \R,$ we define $D^s f$ by its Fourier transform $\widehat{D^s f}(\xi)=|\xi|^s \hat{f}(\xi).$

\section{The fKdV}
 We will  follow closely the strategy  of \cite{A}, which was used to prove the orbital stability of the KdV solitary waves and the (conditional) orbital stability for \eqref{fKdV} in the case $\alpha \geq1$ and related equations.  We just indicate the differences. We will assume in this section that $\frac12<\alpha<1$. We recall that
$$E(u)=\frac{1}{2}\int_\R[|D^{\frac{\alpha}2}u|^2-\frac{1}{3}u^3]dx \quad\text{and}\quad M(u)=\frac{1}{2}\int_\R u^2 dx \, .$$
 
 
 For $q>0$ fixed, we set
 
 \begin{equation}\label{min}
  I_q=\inf_{u\in H^{\frac{\alpha}2}(\R)}\lbrace E(u) \ : \ M(u)=q\rbrace. 
 \end{equation}
 
 We will denote by $G_q$ the set (possibly empty) of minimizers.

\begin{lemma}\label{inf}
For any $q>0$ one has $-\infty<I_q<0.$
\end{lemma}

\begin{proof}
By Sobolev and a standard interpolation inequality, one has for any $\epsilon >0$ and $v \in H^{\frac{\alpha}2}(\mathbb R)$ such that $M(v)=q$,
$$\left|\int_\R v^3dx\right|\leq \|v\|_{\frac16}^3\leq C\|v\|_0^{\frac{3\alpha-1}\alpha}\|v\|_{\frac\alpha2}^{\frac1\alpha}\leq\epsilon\|v\|_{\frac\alpha2}^2+C_\epsilon \|v\|_0^{\frac{2(3\alpha-1)}{2\alpha-1}} \, .$$
Now we write as in the proof of Lemma 3.4 in \cite{A}
\begin{equation}
\begin{split}
E(v)&=E(v)+M(v)-M(v)\\
       &=\frac{1}{2}\int_\R [|D^{\frac{\alpha}2}v|^2+v^2]dx-\frac{1}{6}\int_\R v^3dx-M(v)\\
       &\geq\left(\frac{1}{2}-\frac{\epsilon}{6}\right) \|v\|^2_{\frac{\alpha}2}-q-C'_\epsilon q^{(3\alpha-1)/(2\alpha-1)}\\
       &\geq -q-C'_\epsilon q^{(3\alpha-1)/(2\alpha-1)}>-\infty.
\end{split}
\end{equation}
The fact that $I_q<0$ is easily checked by scaling as in the proof of Lemma 3.4 in \cite{A}.
\end{proof}

So $I_q$ exists and is finite, and the concentration-compactness method is used to prove that it is achieved. A first step is to prove that the minimizing sequences are bounded.

\begin{lemma}\label{bound}
If $\{v_n\}$ is a minimizing sequence for $I_q$ then there exist positive constants $C$ and $\delta$ such that

1. $\|v_n\|_{\frac{\alpha}2}\leq C$ for all $n$ and

2. $|v_n|_3\geq \delta$ for all $n$ sufficiently large.
\end{lemma}
\begin{proof}
Let $\{v_n\}$ be a minimizing sequence for $I_q.$ Firstly, one has by a previous estimate
$$\frac{1}{2}\|v_n\|^2_{\frac{\alpha}2}=E(v_n)+\frac{1}{2}\int_\R v_n^2 dx+\frac{1}{6}\int_\R v_n^3 dx\leq |E(v_n)|+\epsilon\|v_n\|^2_{\frac{\alpha}2}+ C(q) \, ,$$
proving 1. 

In order to prove 2, we argue by contradiction, assuming that for any $k\in\N$ there exists a subsequence $v_{n_k}$ such that $|v_{n_k}|_3\leq 1/k, \forall k.$ This implies
$$I_q=\lim_{k\to \infty}\left(\frac{1}{2}\int_\R |D^{\frac{\alpha}2}v_{n_k}|^2-\frac{1}{6}\int_\R v_{n_k}^3 dx\right)\geq - \lim_{k\to \infty}\frac{1}{6}\int_\R v_{n_k}^3 dx=0 \, ,$$
in contradiction with Lemma \ref{inf}.
\end{proof}

The next step is to prove the sub-additivity of $I_q,$ (see \cite{PLL1,PLL2}). 
\begin{lemma}\label {subadd}
For all $q_1,q_2>0,$ one has
$$ I_{q_1+q_2}<I_{q_1}+I_{q_2} \, .$$
\end{lemma}

\begin{proof}
 As in Lemma 2.4 in \cite{A} the proof follows from a homogeneity argument which we give by sake of completeness. For all $\theta >0$ and $q>0$ we claim that

\begin{equation}\label{claim}
I_{\theta q}=\theta^{(3\alpha-1)/(2\alpha-1)} I_q.
\end{equation}
  
To prove the claim, we set for any function $v\in H^{\frac\alpha2}(\R),$
$$v_\theta(x)=\theta^{\alpha/(2\alpha-1)} v(\theta^{1/(2\alpha -1)}x) \, .$$
Then
$$M(v_\theta)=\theta M(v) , $$
and
$$E(v_\theta)=\theta^{(3\alpha-1)/(2\alpha -1)}E(v) \, .$$
Hence 
\begin{equation}
\begin{split}
I_{\theta  q}&=\inf \lbrace  E(v_\theta) : M(v_\theta)=\theta  q\rbrace\\
&=\inf \lbrace  E(v_\theta) : M(v)= q\rbrace\\
&=\inf \lbrace \theta^{(3\alpha-1)/(2\alpha-1)} E(v) : M(v)=q \rbrace\\
&=\theta^{(3\alpha-1)/(2\alpha-1)} I_q \, .
\end{split}
\end{equation}

It follows then (by choosing $q=1$ and $\theta=q_1+q_2$ in \eqref{claim})  that
\begin{displaymath}
\begin{split}
I_{q_1+q_2}&=(q_1+q_2)^{(3\alpha-1)/(2\alpha-1)}I_1 \\
&<\left(q_1^{(3\alpha-1)/(2\alpha-1)}+q_2^{(3\alpha-1)/(2\alpha-1)}\right)I_1=I_{q_1}+I_{q_2} \, .
\end{split}
\end{displaymath}
  
  
  
  

  










\end{proof}

As usual in the concentration compactness method, we associate to any minimizing sequence $\{v_n\}$ the sequence of nondecreasing functions
 $\mathfrak M_n: \lbrack 0, \infty) \to \lbrack 0,q\rbrack$ defined by 
$$\mathfrak M_n(r)=\sup_{y\in \R}\int_{y-r}^{y+r} |v_n|^2 dx.$$
By an elementary argument, $\lbrace \mathfrak M_n\rbrace$ has a subsequence, still denoted by $\lbrace \mathfrak M_n\rbrace$, which converges uniformly on compact sets to a nondecreasing function $\mathfrak M : \lbrack 0, \infty) \to \lbrack 0,q\rbrack.$ Let
 
 $$\lambda=\lim_{r\to \infty} \mathfrak M(r),\quad \text{so that}\quad 0\leq \lambda\leq q.$$
 
 We will examine successively the three (mutually exclusive) possibilities, $\lambda =q$ (compactness), $\lambda =0$ (vanishing), $0<\lambda<q$ (dichotomy).
 
The compactness case is the good one in virtue of the following lemma. 
\begin{lemma}\label{compact}
 Assume that $\lambda=q.$ Then there exists a sequence of real numbers $\{y_n\}_{n\in \N}$ such that
 
 1. For every $z<q$ there exists $r=r(z)$ such that
 $$\int_{y_n-r}^{y_n+r} |v_n|^2>z$$
 for all sufficiently large $n.$
 
 2. The sequence $\lbrace\tilde{v}_n\rbrace$ defined by
 $$\tilde{v}_n(x)=v_n(x+y_n)\quad \text{for all}\quad x\in \R$$
has a subsequence which converges in $H^{\frac\alpha2}(\R)$ to a function $g\in G_q.$ In particular, $G_q$ is not empty.
 \end{lemma}
  
 \begin{proof}
 The proof is classical and follows exactly that of Lemma 2.5 in \cite{A}, replacing $H^1(\R)$ by $H^{\frac{\alpha}2}(\R).$ 
 \end{proof}
 
 The next technical lemma will be use to prove that vanishing does not occur.
\begin{lemma}\label{tech}
 Suppose that $B>0$ and $\delta >0$ are given. Then there exists $\eta=\eta(B,\delta)$ such that if $v\in H^{\frac\alpha2}(\R), \alpha >1/3,$ with $\|v\|_{\frac\alpha2}\leq B$ and $|v|_3\geq \delta,$ then
$$\sup_{y\in\R}\int_{y-2}^{y+2}|v(x)|^3dx \geq \eta.$$
\end{lemma}
  
  \begin{proof}
  The proof follows exactly that of  Lemmas 3.7, 3.8, 3.9 in \cite{ABS} (see also Lemma 3.3 in \cite{A}) in the case $\alpha =1.$  \end{proof}
  
  
  
  
  
  
  
  

  

  


 
 The following key lemma shows that dichotomy occurs when $0<\lambda<q.$
 \begin{lemma}\label{dicho}
 We still consider a minimizing sequence $\{v_n\}.$  Then for every $\epsilon>0$ there exist $N\in \N$ and sequences $\lbrace g_N, g_{N+1},...\rbrace$ and $\lbrace h_N, h_{N+1},...\rbrace$ of functions in $H^{\frac{\alpha}2}(\R)$  such that for every $n\geq N,$ 
 
 1. $|M(g_n)-\lambda|<\epsilon$
 
 2.  $|M(h_n)-(q-\lambda)|<\epsilon$
 
 3. $E(v_n)\geq E(g_n)+E(h_n)-\epsilon.$
 \end{lemma}
 
 \begin{proof}
 
 Statements 1 and 2 are pretty general and a proof can be found for instance in that of Lemma 2.6 in \cite{A} (see also a sketch of the proof below). Statement 3 is more delicate because of the non locality of $D^\alpha.$
 
 


To prove 3, we follow closely the proof of Lemmas 2.6 and 3.8 in \cite{A}. Let $\phi\in C_0^\infty\lbrack -2,2\rbrack$ be such that $\phi\equiv 1$ on $\lbrack -1,1\rbrack,$ and let $\psi\in C^\infty(\R)$ be such that $\phi^2+\psi^2\equiv 1$ on $\R.$ For each $r\in \R,$ define $\phi_r(x)=\phi(x/r)$ and $\psi_r(x)=\psi(x/r).$ Coming back to the definition of $\mathfrak M$, $\epsilon >0$ being fixed,  for all sufficiently large values of $r$ one has 
$$\lambda-\epsilon<\mathfrak M(r)\leq \mathfrak M(2r)\leq \lambda \, .$$
Such a value of $r$ being fixed, one can choose $N$ so large that
$$\lambda-\epsilon<\mathfrak M_n(r)\leq \mathfrak M_n(2r)\leq \lambda+\epsilon$$
for all $n\geq N.$ Hence for each $n\geq N,$ one can find $y_n$ such that
\begin{equation}\label{eq1}
\int_{y_n-r}^{y_n+r} |v_n|^2 dx >\lambda -\epsilon
\end{equation} 
and
\begin{equation}\label{eq2}
\int_{y_n-2r}^{y_n+2r} |v_n|^2 dx <\lambda +\epsilon
\end{equation}

Define $g_n(x)=\phi_r(x-y_n)v_n(x)$ and $h_n(x)=\psi_r(x-y_n)v_n(x).$ Clearly $g_n$ and $h_n$ satisfy statements 1 and 2.


We now write
\begin{equation}
\begin{split}
E(g_n)+E(h_n)&=\frac{1}{2}\left\lbrack \int\phi_r^2v_nD^\alpha v_n dx+\int\phi_rv_n\lbrack D^\alpha, \phi_r\rbrack v_n dx \right\rbrack\\
&+\frac{1}{2}\left\lbrack \int\psi_r^2v_nD^\alpha v_n dx+\int\psi_rv_n\lbrack D^\alpha, \psi_r\rbrack v_n dx \right\rbrack\\
&-\frac{1}{6}\int \phi_r^2v_n^3 dx -\frac{1}{6}\int \psi_r^2v_n^3 dx\\
&+\frac{1}{6}\int (\phi_r^2-\phi_r^3)v_n^3 dx+\frac{1}{6}\int (\psi_r^2-\psi_r^3)v_n^3 dx\\
&=E(v_n)+\int\phi_rv_n\lbrack D^\alpha, \phi_r\rbrack v_n dx \int\psi_rv_n\lbrack D^\alpha, \psi_r\rbrack v_n dx\\
&+\frac{1}{6}\int (\phi_r^2-\phi_r^3)v_n^3 dx+\frac{1}{6}\int (\psi_r^2-\psi_r^3)v_n^3 dx \, ,
\end{split}
\end{equation}
where we have used that $\phi^2+\psi^2\equiv 1.$ 

As in \cite{A} we want to prove that the sum of the two commutators is $O(1/r^\beta)$ for some $\beta >0$ and that the sum of the two other terms is $O(\epsilon).$ For the later this is exactly as in \cite{A}. For the commutator, since in his case $\alpha =1$, Albert uses that $|\lbrack |D|,\theta\rbrack f|_2\leq C|\theta '|_\infty |f|_2$ and this is fine since $|\phi'_r|_\infty=1/r|\phi'|_\infty.$ 
 
For $\alpha<1,$ we will use instead the fractional Leibniz rule of Kenig, Ponce and Vega (\textit{cf} Theorem A.8 and A.12 in the appendix of \cite{KPV}).
\begin{lemma}[Fractional Leibniz Rule]
Let $0<\alpha<1$, $1<p, \, p_1, \, p_2<+\infty$ and $\alpha_1, \, \alpha_2 \in [0,\alpha]$ be such that 
$\frac1p=\frac1{p_1}+\frac1{p_2}$ and $\alpha=\alpha_1+\alpha_2$. Then 
\begin{equation} \label{LR.1}
\big| D^{\alpha}(fg)-fD^{\alpha}g-gD^{\alpha}f\big|_p \lesssim |D^{\alpha_1}f|_{p_1}|D^{\alpha_2}g|_{p_2} \, .
\end{equation}
Moreover if $\alpha_1=0$, then $p_1=+\infty$ is allowed. 
\end{lemma}


First, we estimate $|[D^{\alpha},\phi_r]v_n |_2$.
Observe that
\begin{equation} \label{eq1}
|[D^{\alpha},\phi_r]v_n |_2 \le \big| D^{\alpha}(\phi_r v_n)-\phi_rD^{\alpha}v_n-v_nD^{\alpha}\phi_r\big|_2+|v_nD^{\alpha}\phi_r|_2 
\end{equation}
 Thus, by using \eqref{LR.1} with $f=v_n$, $g=\phi_r$, $p=2$, $p_1=p_2=4$ and $\alpha_2=\alpha$, $\alpha_1=0$, we get that
\begin{equation} \label{eq2}
|[D^{\alpha},\phi_r]v_n |_2 
\lesssim |v_n|_4 |D^{\alpha}(\phi_r)|_4 \, .
\end{equation}

On the one hand due to the Sobolev embedding $H^{\frac14}(\mathbb R) \hookrightarrow L^4(\mathbb R)$ and the fact that $\{v_n\}$ is bounded in $H^{\frac{\alpha}2}(\mathbb R)$ with $\frac{\alpha}2 >\frac14$, there exists $C>0$ such that 
\begin{equation} \label{eq3}
|v_n|_4 \le C \, .
\end{equation}
On the other hand, a direct computation yields 
\begin{equation} \label{eq4}
|D^{\alpha}(\phi_r)|_4 =r^{\frac14-\alpha}|D^{\alpha}\phi|_4=\mathcal{O}(r^{\frac14-\alpha}) \, ,
\end{equation}
since $\phi \in C_0^{\infty}(\mathbb R) \subset \mathcal{S}(\mathbb R)$.
Thus, we conclude gathering \eqref{eq2}--\eqref{eq4} that 
\begin{equation} \label{eq5}
|[D^{\alpha},\phi_r]v_n |_2 =\mathcal{O}(r^{\frac14-\alpha}) \, ,
\end{equation}
which is fine since $\alpha>\frac12$.

We use the same strategy to estimate $|[D^{\alpha},\psi_r]v_n |_2$. From the definition of $\psi$, we have  $\phi^2+\psi^2=1$, so that we can write 
$$\psi=1-\chi \quad \text{where} \quad \chi=1-\sqrt{1-\phi^2} \in C_0^{\infty}(\mathbb R) \subset \mathcal{S}(\mathbb R). $$
Moreover, it holds that $\big(D^{\alpha}(1)\big)^{\wedge}(\xi)=c|\xi|^{\alpha}\delta_0=0$ in $\mathcal{S}'(\mathbb R)$ . Then
\begin{displaymath}
|D^{\alpha}(\psi_r)|_4=|D^{\alpha}(\chi_r)|_4 =r^{\frac14-\alpha}|D^{\alpha}\chi|_4=\mathcal{O}(r^{\frac14-\alpha}) \, .
\end{displaymath}
Therefore, we conclude arguing as above that
\begin{equation} \label{eq5}
|[D^{\alpha},\psi_r]v_n |_2 =\mathcal{O}(r^{\frac14-\alpha}) \, .
\end{equation}

Finally we have established that 
$$E(g_n)+E(h_n)=E(v_n)+\mathcal {O}(r^{\frac14-\alpha})+\mathcal {O}(\epsilon),$$
which achieves the proof of 3.
\end{proof}
 
 As in \cite{A} Corollary 2.7, one deduces from Lemma \ref{dicho}
 \begin{corollary}\label{cor1}
 If $0<\lambda <q,$ then
 $$I_q\geq I_\lambda+I_{q-\lambda}.$$
\end{corollary}

Corollary \ref{cor1} shows why dichotomy cannot hold.
We now prove that vanishing does not occur.
  \begin{lemma}\label{vanish}
 For every minimizing sequence,  $\lambda>0.$
 \end{lemma}
 
 \begin{proof}  
 By Lemmas \ref{bound} and \ref{tech} there exist $\eta >0$ and a sequence $\lbrace y_n\rbrace$ such that 
 $$\int_{y_n-2}^{y_n+2} |v_n|^3 dx\geq \eta \quad \text{for all}\;n.$$
 Hence,
\begin{equation}
 \begin{split}
 \eta&\leq \left(\int_{y_n-2}^{y_n+2}|v_n|^2dx\right)^{1/2}\left(\int_{y_n-2}^{y_n+2}|v_n|^4dx\right)^{1/2}\\
 &\leq \left(\int_{y_n-2}^{y_n+2}|v_n|^2dx\right)^{1/2}\left(\int_\R|v_n|^4dx\right)^{1/2}\\
 &\leq C\left(\int_{y_n-2}^{y_n+2}|v_n|^2dx\right)^{1/2},
 \end{split}
 \end{equation}
 where we have used the embedding $H^{\frac{\alpha}2}(\R)\hookrightarrow L^4(\R)$ when $\alpha \geq \frac{1}{2}.$
 
Thus
$$\lambda=\lim_{r\to\infty}M(r) \geq M(2)=\lim_{n\to \infty}M_n(2)\geq \frac{\eta}{C}>0.$$
 \end{proof}
 
 We can now state and prove our main result. We first recall (see \cite{LPS2}) that the Cauchy problem for \eqref{fKdV} is locally well-posed in $H^s(\R),\; s> s_\alpha=\frac{3}{2}-\frac{3\alpha}{8}$ in the sense that for any $u_0\in H^s(\R)$ with $s$ as above, there exists a maximal time of existence $T_s \in (0,+\infty]$ and a unique solution $u$ to \eqref{fKdV} such that $u\in C(\lbrack 0,T_s);H^s(\R))$ satisfying $Q(u(\cdot,t))=Q(u_0)$ and $E(u(\cdot,t))=E(u_0), \, t \in\lbrack 0, T_s).$
 
 \begin{theorem}\label{main}
 Let $\frac{1}{2}<\alpha <1.$
 
1. For every $q>0$ there exists a nonempty set $G_q$ of minimizers of \eqref{min} consisting of solitary waves of  \eqref{fKdV} with positive velocity. Moreover, if $\{v_n\}$ is a minimizing sequence for $I_q,$ then the following assertions are true.
 
 2. There exist a sequence $\lbrace y_1,y_2,...\rbrace$ and an element $g\in G_q$ such that $\{v_n(\cdot+y_n)\}$ has a subsequence converging strongly in $H^{\frac{\alpha}2}(\R)$ to $g.$
 
 3. $$\lim_{n\to \infty}\inf_{g\in G_q, y\in \R}\|v_n(\cdot+y)-g\|_{\frac{\alpha}2}=0.$$
 
 4. $$\lim_{n\to \infty}\inf_{g\in G_q}\|v_n-g\|_{\frac{\alpha}2}=0.$$
 
 5. The set $G_q$ is stable in the following sense. For any $ \epsilon >0$ there exists $\delta >0$ such that if $u_0\in H^s(\R), s>s_\alpha,$ with
 $$\inf_{g\in G_q}\|u_0-g\|_{\frac{\alpha}2}<\delta,$$ 
 then the corresponding solution $u$ emanating from $u_0$ of \eqref{fKdV} satisfies 
 $$\inf_{g\in G_q}\|u(\cdot,t)-g\|_{\frac{\alpha}2}<\epsilon,\quad \forall \,  0<t<T_s \, .$$ 
 \end{theorem}
 
\begin{proof}
The proof is a classical application of the concentration-compactness method. By Lemmas \ref{subadd}, \ref{dicho}, \ref{vanish} and Corollary \ref{cor1} we deduce that $\lambda =q.$

We prove 2 by contradiction, assuming that there exist a subsequence $\lbrace v_{n_k}\rbrace$ of $\lbrace v_n\rbrace$ and $\epsilon >0$ such that 
$$\inf_{g\in G_q, y\in \R}\|v_{n_k}(\cdot+y)-g\|_{\frac{\alpha}2}\geq \epsilon$$
for all $k\in \N.$ Since $\{v_{n_k}\}$ is also a minimizing sequence for $I_q,$ statement 1 implies that there exist a sequence  $\lbrace y_k\rbrace$ and $g_0\in G_q$ such that 
$$\liminf_{k\to \infty} \|v_{n_k}(\cdot + y_k)-g_0\|_{\frac{\alpha}2}=0,$$
and this contradiction proves 2.

The stability statement 5 is classically proven by contradiction from 4.
\end{proof}

We now relate the set $G_q$ of minimizers to $I_q$ to the {\it ground states} as  defined in \cite{FL},  Definition 2.1.

 
\begin{definition} \label{GroundState}\cite{FL}

Let $\alpha>\frac13$. A {\it ground state solution} of
\begin{equation} \label{FrLe.2}
D^{\frac{\alpha}2}Q+Q -Q^2=0 \, ,
\end{equation}
is a positive and even solution that solves the minimization problem
\begin{equation} \label{FrLe.3}
J^\alpha (Q)=\inf \big\{ J^\alpha (u) \ : \ u\in H^{\frac{\alpha}2}(\R)\setminus \lbrace 0\rbrace\big\} \, ,
\end{equation}
where $J^{\alpha}$ is the Weinstein functional defined by
\begin{equation} \label{FrLe.4}
J^\alpha (u)= \left(\int_\R|u|^3dx\right)^{-1}\left(\int_\R |D^{\frac{\alpha}2} u|^2 dx\right)^{\frac1{2\alpha}}\left(\int_\R u^2 dx\right)^{\frac{3\alpha -1}{2\alpha}} \, .
\end{equation}
\end{definition}

\begin{lemma} \label {FrLe}
Let $q>0$ and $\frac{1}{2}<\alpha <1$. Any minimizer $\psi$ of $I_q$ writes  
\begin{equation} \label{FrLe.1}
\psi=c Q\big(c^{\frac1\alpha}(\cdot+y)\big)
\end{equation} 
for some $y\in \R$ and $c>0$ chosen  to ensure that $\frac{1}{2}\int_\R \psi^2 dx=q$ holds and $Q$ is a  ground state solution of \eqref{FrLe.2}.
\end{lemma}

In order to prove Lemma \ref{FrLe}, we recall the fundamental result\footnote{stated here in our context.} of Frank and Lenzmann in Theorem 2.4 of \cite{FL}. 
\begin{theorem} \label{FrLetheo}
Let $\alpha>\frac13$. Then, the ground state solution $Q=Q(|x|)>0$ of equation \eqref{FrLe.2} is unique. 

Furthermore, every minimizer $v \in H^{\frac{\alpha}2}(\mathbb R)$ for the Weinstein functional $J^{\alpha}$ defined in \eqref{FrLe.4} is of the form $v=\beta Q(\lambda(\cdot+y))$ for some $\beta \in \mathbb C$, $\beta \neq 0$, $\lambda>0$ and $y \in \mathbb R$.
\end{theorem}

\begin{proof}[Proof of Lemma \ref{FrLe}] Assume that $q>0$ is fixed. Let $Q$ be a ground state of \eqref{FrLe.2} defined as above. Observe that for any $c>0$, $Q_c=cQ(c^{\frac1\alpha}\cdot)$ is a solution to \eqref{sol}. It follows from\eqref{EnergyIdentity} and \eqref{PohojaevIdentity} that 
\begin{equation} \label{FrLe.5} 
\int_{\mathbb R} |D^{\frac{\alpha}2}Q_c|^2dx=\frac{c}{3\alpha-1} \int_{\mathbb R} Q_c^2 dx
\end{equation}
and 
\begin{equation} \label{FrLe.6}
\int_{\mathbb R} Q_c^3dx=\frac{6\alpha c}{3\alpha-1}  \int_{\mathbb R} Q_c^2 dx \, .
\end{equation}
Therefore, a straightforward computation gives that 
\begin{equation} \label{FrLe.7} 
J^{\alpha}(Q_c)=\frac{(3\alpha-1)^{1-\frac1{2\alpha}}}{6\alpha}  c^{\frac1{2\alpha}-1} \|Q_c\|_{L^2}
=\frac{(3\alpha-1)^{1-\frac1{2\alpha}}}{6\alpha}   \|Q\|_{L^2} \, .
\end{equation}
Note in particular that the minimum of $J^{\alpha}$ is attained for every $Q_c$ with $c>0$. Moreover, we choose $c_{\star}>0$ such that 
\begin{equation} \label{FrLe.7a}
M(Q_{c_{\star}})=q \quad \Leftrightarrow \quad c_{\star}=\left(\frac{2q}{\|Q\|_{L^2}^2} \right)^{\frac{\alpha}{2\alpha-1}} \, .
\end{equation}
Another easy computation yields 
\begin{equation} \label{FrLe.7b}
E(Q_{c_{\star}})=\frac{c_{\star}}{3\alpha-1}(\frac12-\alpha) 2q \, .
\end{equation}

Now, let $\psi \in G_q$, \textit{i.e.} $\psi$ is a minimizer of $I_q$. By the Lagrange multipliers theory, there exists $\theta_q \in \mathbb R$ such that 
\begin{equation} \label{FrLe.8} 
D^{\alpha}\psi-\frac12\psi^2+\theta_q \psi=0 \, .
\end{equation}

By using the energy and Pohojaev identities, we deduce exactly as in \eqref{FrLe.5} and \eqref{FrLe.6} that
\begin{equation} \label{FrLe.9} 
\int_{\mathbb R} |D^{\frac{\alpha}2}\psi|^2dx=\frac{\theta_q}{3\alpha-1} \int_{\mathbb R} \psi^2 dx=\frac{2q\theta_q}{3\alpha-1} 
\end{equation}
and 
\begin{equation} \label{FrLe.9a}
\int_{\mathbb R} \psi^3dx=\frac{6\alpha \theta_q}{3\alpha-1}  \int_{\mathbb R} \psi^2 dx= \frac{12q\alpha \theta_q}{3\alpha-1}  \, .
\end{equation}
Identities \eqref{FrLe.9} and  \eqref{FrLe.9a} imply in particular that $\theta_q>0$ and $\int_{\mathbb R} \psi^3dx>0$, since $\alpha>\frac13$.

Next, we prove that $\psi$ must be positive. Indeed, recall that 
\begin{displaymath} 
|D^{\frac\alpha2}(|\psi|)|_2 \le |D^{\frac\alpha2}\psi|_2 \, ,
\end{displaymath}
for $\frac12<\alpha<1$. This claim follows for example from estimate (2.10) in \cite{FQT}. Therefore, we deduce that 
$E(|\psi|) \le E(\psi)$ and $M(|\psi|)=q$, since we also have 
\begin{equation} \label{FrLe.9b}
\int_{\mathbb R} \psi^3dx=\left| \int_{\mathbb R} \psi^3dx \right| \le \int_{\mathbb R} |\psi|^3dx \, .
\end{equation}
Moreover, if $\psi$ is not positive on $\mathbb R$, then the inequality in \eqref{FrLe.9b} is strict, so that $E(|\psi|)<E(\psi)$, which is a contradiction since $\psi \in G_q$.

We compute as above that 
\begin{displaymath} 
J^{\alpha}(\psi)=\frac{(2q)^{\frac12}}{\theta_q^{1-\frac1{2\alpha}}} \frac{(3\alpha-1)^{1-\frac1{2\alpha}}}{6\alpha} \, .
\end{displaymath}
On the one hand, since $J^{\alpha}(\psi) \ge J^{\alpha}(Q_{c_{\star}})$, it follows from \eqref{FrLe.7} and the definition of $c_{\star}$ in \eqref{FrLe.7a} that
\begin{equation} \label{FrLe.10}
\theta_q \le c_{\star} \, .
\end{equation}
On the other hand, another simple computation gives  that
\begin{displaymath} 
E(\psi)=\frac{\theta_q}{3\alpha-1}(\frac12-\alpha) 2q \, .
\end{displaymath}
Since $\psi \in G_q$, we have $E(\psi) \le E(Q_{c_{\star}})$ which implies from \eqref{FrLe.7b} that 
\begin{equation} \label{FrLe.11}
\theta_q \ge c_{\star} \, ,
\end{equation}
in the case $\alpha>\frac12$. We conclude gathering \eqref{FrLe.10} and \eqref{FrLe.11} that 
\begin{equation} \label{FrLe.12}
\theta_q=c_{\star} \, .
\end{equation}
Therefore $J^{\alpha}(\psi)=J^{\alpha}(Q)$ and we conclude from the uniqueness result in Theorem \ref{FrLetheo} that $\psi=Q_{c_{\star}}(\cdot-y)$, for some $y \in \mathbb R$.
 \end{proof}
 
 Finally, as a consequence of Theorem \ref{main} and Lemma \ref{FrLe}, we get the orbital stability of the ground states. 
 \begin{theorem} \label{OrbStab} 
 Let $\frac12<\alpha<1$, $c>0$ and $Q_c=cQ(c^{\frac1{\alpha}}\cdot)$, where $Q$ is the ground state solution of \eqref{FrLe.2}.
 For every $\epsilon>0$, there exists $\delta>0$ such that if $u_0 \in H^s(\mathbb R)$, $s> s_\alpha=\frac{3}{2}-\frac{3\alpha}{8}$, satisfy 
 \begin{equation} \label{OrbStab.1}
 \|u_0-Q_c\|_{\frac{\alpha}2} < \alpha \, ,
 \end{equation}
  then the corresponding solution $u$ emanating from $u_0$ of \eqref{fKdV} satisfies 
 \begin{equation} \label{OrbStab.2}
 \inf_{y \in \mathbb R}\|u(\cdot,t)-Q_c(\cdot+y)\|_{\frac{\alpha}2}<\epsilon
 \end{equation}
for all $t \in [0,T_s)$, where $T_s$ is the maximal time of existence of $u$.
 \end{theorem}

\begin{remark}
The (orbital) stability statement in Theorem \ref{main} is a conditional one. It would become {\it unconditional} provided one establishes the {\it global} well-posedness of the Cauchy problem for data in the  space $H^s(\R),$ $s\leq \frac{\alpha}{2}$ when $\alpha>1/2.$  As previously mentioned, the best known result (\cite{LPS2}) establishes the local well-posedness of the Cauchy problem in $H^s(\R), s>\frac{3}{2}-\frac{3\alpha}{8},$ for any $\alpha >0.$ On the other hand  it is proved in \cite{GV} that  when $\alpha>\frac{1}{2},$ global  weak $L^2$ solutions exist, as well as   global $H^{\frac{\alpha}2}$ weak solutions, uniqueness being unknown. Also  the numerical simulations of \cite{KS} suggest that no finite time blow-up occurs when $\alpha >\frac{1}{2}$, at least for smooth and localized initial data. Recall that when $1< \alpha <2$ the Cauchy problem is globally well-posed for initial data in $L^2(\R)$ (\cite{HeIoKeKo}).
\end{remark}

\begin{remark}\label{linstab}
It has been established in \cite{KaSt} that the ground state is spectrally stable when $\alpha>\frac{1}{2}.$
\end{remark}

\begin{remark}
The results above extend {\it mutatis mutandis} to the generalized fractional KdV equation
\begin{equation}
\label{GfKdV}
u_t+u^pu_x-D^\alpha u_x=0,\quad u(.,0)=u_0
\end{equation}
in the $L^2$ subcritical case, that is $\alpha>\frac{p}{2}.$
\end{remark}

\begin{remark}
It would be interesting to prove the {\it asymptotic stability} of the ground states of \eqref{fKdV} and also the existence (and stability) of {\it multisoliton solutions} of \eqref{fKdV}. Such solutions have been proven to exist and to be stable (in the subcritical case) for the generalized Korteweg-de Vries equations (see \cite{Ma, MMT}).\end{remark}

\subsection{Remarks on instability}
Instability of solitary wave solutions  of the gKV equation
\begin{equation}\label{gKdV}
u_t+u_x+u^pu_x+u_{xxx}=0,
\end{equation}
has been established in \cite{BSS} when $p> 4$ and in \cite{MM2} for $p=4.$
  \vspace{0.3cm}   
  
The mechanism of instability and the links with  finite type blow-up are now well understood in the $L^2$ critical case $p=4$ (see \cite{MM}, \cite{MMR1, MMR2, MMR3} for theoretical studies and \cite{KP} for numerical simulations).
  
   However a precise description of the instability in the super critical case $p>4$ and in particular the proof of finite type blow-up are not known. Note that the link between instability and finite type blow-up is strongly suggested by the numerical simulations in \cite{BDKM} and \cite{KP} (where the $L^2$ critical case is also considered).
  
  We now turn to the expected instability  of the fKdV solitary waves when $\frac{1}{3}<\alpha\leq\frac{1}{2}.$ The numerical simulations in \cite{KS2} suggest that the instability mechanism is via finite time blow-up, similar to the KdV $L^2$ critical when $\alpha =1/2$ and to the KdV $L^2$ supercritical case when $ 1/3<\alpha <1/2.$ Proving such results appears to be out of reach, and we should restrict to the mere instability proof. As in \cite{BSS} the first step is to give a sense to the formal conserved quantity
  
  \begin {equation}\label{I}
  I(u)=\int_\R u dx.
  \end{equation}

Exactly as in  Proposition 2.1 in \cite{BSS}, one checks that if $u_0\in H^s(\R), s\geq 1+\alpha$ is such that $\int_{-\infty}^{\infty} u_0(x) dx$ converges as a generalized Riemann integral, then $I(u(t))$ converges for any $t\in [0,T_s(u_0))$ and is constant, where $T_s(u_0)$ is the lifespan of the solution $u$ of the corresponding Cauchy problem.

Again as in \cite{BSS} one has to estimate how fast the tail of $I(u)$ near infinity grows with $t$. This cannot be deduce directly from Theorem 2.2 in \cite{BSS} since 
$$G_\alpha(x)=\int_{-\infty}^\infty e^{i(x\xi-\xi|\xi|^\alpha)} d\xi$$
is not  a bounded function of $x$  when $\alpha<1. $

 Actually, (see \cite{SSS}), $G_\alpha(x)=O(x^{-(\alpha+2)})$ as $x\to +\infty$ and oscillates when $x\to -\infty,$ growing as $|x|^{(1-\alpha)/2\alpha}.$ 
 






In order to prove the equivalent of  Theorem 2.2 in \cite{BSS}, one would need to impose a (one sided) decay property to $u_0$ insuring that the resulting solution of the Cauchy problem decays sufficiently to the left to compensate the growth of the fundamental solution.
  
  \section {The fBBM equation}
  
  As previously noticed an alternative to the toy model \eqref{fKdV} is the fractional Benjamin-Bona-Mahony equation (fBBM) \eqref{fBBM}.
  
  
  A solitary wave solution $u_c(x,t)=\phi(x-ct), c>0$ of \eqref {fBBM} satisfies the equation
  
   \begin{equation}\label{SWfBBM}
   (c+D^\alpha)u-\frac{u^2}{2}=0.
   \end{equation}
   
   Existence and stability issues for \eqref{SWfBBM} have been considered in \cite{Z} when $\alpha >1$ but the proofs therein extend readily to the  case  $\alpha<1.$  

More precisely, Zeng considers the minimization problem

\begin{equation}\label{min}
I_q=\inf\lbrace F(u) : u\in H^{\frac\alpha2}(\R)\;\text{and}\; N(u)=q\rbrace,
\end{equation}
where
$$F(u)=\int_\R (u^2+|D^{\alpha/2}u|^2)$$
and 
$$N(u)=\int_\R \left(\frac{u^2}{2}+\frac{u^3}{6}\right).$$
He thus considers the set of {\it ground state solutions} of \eqref{fBBM}, that is
$$G_q= \lbrace u\in H^{\frac\alpha2}(\R) : N(u)=q\;\text{and}\; F(u)=I_q \rbrace.$$

The results established in \cite{Z} for $\alpha\geq 1$ and general nonlinearities $u^pu_x$ extends without any noticeable change in our case and imply the following theorem.
\begin{theorem}\label {Zeng}\

1. Assume that $\frac{1}{2}<\alpha <1.$ Then the set $G_q$ is not empty and orbitally stable in $H^{\frac\alpha2}(\mathbb R).$

2. Assume that $\frac{1}{3}<\alpha <\frac{1}{2}.$Then there exists $q_0=q_0(\alpha)$ such that for all $q>q_0,$ $G_q$ is not empty and orbitally stable in $H^{\frac\alpha2}(\mathbb R).$
\end{theorem}

\begin{remark}
1. Again, the orbital stability results in Theorem \ref{Zeng} are conditional ones. A complete one would necessitate to prove a {\it global} well-posedness for the Cauchy problem associated to \eqref{fBBM}, when $\alpha>1/3.$ Due to the invariance of the $H^{\frac\alpha2}(\mathbb R)$ norm, it would be sufficient to get a  {\it local} well-posedness result in the same space. We recall that the best known result so far is given in \cite{LPS2} where local well -posedness  is proven for initial data in $H^s(\R), s>\frac{3}{2}-\alpha.$

Note that  the conservation of $E(u)$ implies by standard compacteness methods the global existence of weak solutions in $H^{\frac\alpha2}(\mathbb R),$ without uniqueness.

It is worth noticing that the numerical simulations in \cite{KS} suggest that a finite type blow-up may occur when $0<\alpha\leq \frac{1}{3}$ but not when $\alpha >\frac{1}{3}.$

2. In the case of the generalized BBM equation \eqref{gBBM}, the critical value $q_0$ is associated to a critical velocity for the solitary waves, \lq\lq fast\rq\rq \, solitary waves are stable (see more details below). This fact relies strongly on the {\it explicit} formulas for the solitary waves. No such link seems to be known for fractional BBM equations.

\end{remark}

As noticed in \cite{BMcR} for the generalized BBM equation
\begin{equation}\label{gBBM}
u_t+u_x+u^pu_x-u_{xxt}=0,
\end{equation}
the stability theory of solitary waves is \lq\lq a little more complex\rq\rq \, than for the corresponding generalized KdV equation \eqref{gKdV} for which any solitary wave of {\it arbitrary} positive velocity is unstable when $p\geq 4.$

In fact (see \cite{SS}) solitary waves of \eqref{gBBM} of arbitrary positive velocity are stable when $p<4$ but when $p\geq 4$ there exists $c*=c*(p)$
such that the solitary waves of velocity $c<c*$ are unstable while those of velocity $c>c*$ are stable.

Furthermore the mechanism of instability is different since the Cauchy problem for gBBM is globally well posed in $H^1(\R)$ for any $p.$ The numerical simulations in \cite{BMcR} suggest that an unstable solitary wave will jump to a stable, faster one. No rigorous proof of this fact exists to our knowledge.

 Instability results for generalized  fBBM type equations are provided in \cite{SS} when $\alpha\geq 1$ in our notations. The proof does not extend easily to the case $\alpha<1$ (they use properties of the multiplier $
m(\xi)=1+|\xi|^\alpha$ that are no more valid when $\alpha<1).$ 

\section{Remarks on the KP case}

We  consider  now briefly the KP I version of \eqref{fKdV}, that is
\begin{equation} \label{fKPI}
u_t+uu_x-D_x^\alpha u_x+\epsilon\partial_x^{-1} u_{yy}=0,\quad \text{in}\; \R^2\times \R_+,\quad u(\cdot,0)=u_0,
\end{equation}
where $\epsilon =1$ corresponds to the fKP II equation and $\epsilon =-1$ to the fKP I equation. Here $D^{\alpha}_x$ denotes the Riesz potential of order $-\alpha$ in the $x$ direction, \textit{i.e.} $D^{\alpha}_x$ is defined via Fourier transform by $\big(D^{\alpha}_xf\big)^{\wedge}(\xi,\eta)=|\xi|^{\alpha}\widehat{f}(\xi,\eta)$.

In addition to the $L^2$ norm, \eqref{fKPI} conserves formally the energy (Hamiltonian)

\begin{equation}\label{HamfKP}
H_\alpha(u)=\int_{\R^2} (\frac{1}{2}|D_x^{\frac\alpha2} u|^2-\epsilon\frac{1}{2}|\partial_x^{-1}u_y|^2-\frac{1}{6} u^3).
\end{equation}
The corresponding energy space is 
$$Y_\alpha= \lbrace u\in L^2(\R^2) \ :  \ D^{\frac\alpha2}_x u, \ \partial_x^{-1}u_y\in L^2(\R^2)\rbrace.$$

The first question is to which values of $\alpha$ correspond to the $L^2$ and the energy critical cases? 

For the generalized KP-I equations 
\begin{equation}\label{gKPI}
u_t+u^pu_x+u_{xxx}-\partial_x^{-1}u_{yy}=0,
 \end{equation} 
 the corresponding  values of $p$ are respectively $p=4/3$ and $p=4$ (see \cite{deBS, deBS1,deBS2}). 
 
 One checks readily that the transformation
 $$u_\lambda(x,y,t)=\lambda^\alpha u(\lambda x, \lambda ^{\frac{\alpha+2}{2}} y, \lambda^{\alpha+1} t)$$ leaves \eqref{fKPI} invariant. 
 
 Moreover, $|u_\lambda|_2=\lambda^{\frac{3\alpha -4}{4}}|u|_2$, so that $\alpha=\frac{4}{3}$ is the $L^2$ critical exponent.

The energy critical value of $\alpha$ is obviously related to the non existence of localized solitary waves. One has :

\begin{proposition}\label{nonex}
Assume that $0<\alpha\leq \frac{4}{5} $ when $\epsilon =-1$ or that $\alpha$ is arbitrary when $\epsilon =1.$Then \eqref{fKPI} does not possess non trivial  solitary waves in the space $Y_\alpha \cap L^3(\R^2).$
\end{proposition}

\begin{proof}
It is handy to write \eqref{fKPI} as
\begin{equation} \label{BBMBsq}
\left\{ \begin{array}{l}   -cu_x+uu_x-D_x^\alpha u_x+\epsilon v_y=0 \\
v_x=u_y,
\end{array} \right.
\end{equation}

Adapting the method in \cite{deBS}, we multiply successively the first equation by $xu$ and $yv.$ After some integrations by parts (which can be justified by a standard truncation in space procedure and a truncation of low frequencies as in \cite{Mol}) one obtains the two identities:

\begin{equation}\label{po1}
\int_{\R^2} \left(\frac{c}{2}u^2-\frac{1}{3}u^3+\epsilon \frac{1}{2}v^2+\frac{\alpha +1}{2}|D_x^{\frac\alpha2}u|^2\right)=0,
\end{equation}

\begin{equation}\label{po2}
\int_{\R^2} \left(-\frac{c}{2}u^2+\frac{1}{6}u^3-\epsilon \frac{1}{2}v^2-\frac{1}{2}|D_x^{\frac\alpha2}u|^2\right)=0.
\end{equation}

On the other hand, the energy identity yields

\begin{equation}\label{energ}
\int_{\R^2} \left(-cu^2+\frac{1}{2}u^3+\epsilon v^2-|D_x^{\frac\alpha2}u|^2\right)=0.
\end{equation}

Substracting \eqref{po2} from \eqref{po1} the cubic term from \eqref{po1} yields

\begin{equation}
\int_{\R^2}\left(cu^2-\frac{1}{2} u^3+\epsilon v^2+\frac{\alpha +2}{2}|D_x^{\frac\alpha2}u|^2\right)=0,
\end{equation}
and adding with \eqref{energ} we obtain
\begin{equation}\label{defoc}
\int_{\R^2}\left(2\epsilon v^2+\frac{\alpha}{2}|D_x^{\frac\alpha2}u|^2\right)=0,
\end{equation}
proving that no solitary wave exists, whatever $\alpha$ in the defocusing case $\epsilon =1.$

In the focusing, fKP I, case $\epsilon=-1,$ we use \eqref{defoc} successively in \eqref{po1} and \eqref{energ} to get 
\begin{equation}\label{po3}
\int_{\R^2}\left(\frac{c}{2} u^2-\frac{1}{3}u^3+ \frac{3\alpha+4}{2\alpha}v^2\right)=0,
\end{equation}
and
\begin{equation}\label{po4}
\int_{\R^2}\left(-c u^2+\frac{1}{2}u^3- \frac{\alpha+4}{\alpha}v^2\right)=0.
\end{equation}

Eliminating $v$ we obtain
\begin{equation}\label{po5}
\int_{\R^2}\left(c\alpha  u^2+\frac{4-5\alpha}{12}u^3\right)=0.
\end{equation}

On the other hand, adding \eqref{po1} and \eqref{po2} yields
\begin{equation}\label{po6}
\frac{1}{3}\int_{\R^2}u^3=\alpha\int_{\R^2}|D_x^{\frac\alpha2}u|^2,
\end{equation}
which with \eqref{po5} implies
\begin{equation}\label{po7}
\int_{\R^2}\left(cu^2+\frac{4-5\alpha}{4}|D_x^{\frac\alpha2}u|^2\right)=0,
\end{equation}
which proves that no solitary waves exist in this case when $\alpha\leq \frac{4}{5}.$
\end{proof}

To go further it might be useful to consider the situation for the generalized KPI equation \eqref{gKPI}. Existence of solitary waves is established  in \cite{deBS}  in the energy subcritical case $1\leq p<4,$ by solving the variational problem $I_\lambda$ consisting in minimizing the energy norm with the constraint

$$\int_{\R^2} u^{p+2}=\lambda.$$

To define the notion of {\it ground state} for \eqref{gKPI}, we introduce  the energy

$$E_{KP}(\psi) = \frac{1}{2} \int_{\R^2} (\partial_x \psi)^2 + \frac{1}{2} \int_{\R^2} (\partial_x^{-1}\partial_y \psi)^2 - \frac{1}{2(p+2)} \int_{\R^2} \psi^{p+2},$$
and we define the action

$$S(N) = E_{KP}(N) + \frac{c}{2} \int_{\R^2} N^2.$$

We term {\it ground state}, a solitary wave $N$ which minimizes the action $S$ among all finite energy non-constant solitary waves of speed $c$ of \eqref{gKPI} (see \cite{deBS} for more details). It is proven in \cite{deBS} that when $1\leq p<4$, the solutions of the minimization problem $I_\lambda$ are ground states. Moreover (see \cite{deBS2}) , when $1\leq p<\frac{4}{3},$ the ground states are minimizers of the Hamiltonian  $E_{KP}$ with prescribed mass ($L^2$ norm). This implies (by an argument {\it \`a la Cazenave-Lions})  the orbital stability of the set of ground states (see also \cite{LW}). The {\it uniqueness}, up to the trivial symmetries of the ground states is a challenging open question.\footnote{The stability result in \cite{deBS2} is a {\it conditional one} when $p\neq 1$ by lack of the {\it global} well-posedness of the corresponding Cauchy problem. Recall that the Cauchy problem for the KPI equation itself ($p=1$) is globally well-posed in appropriate spaces!
 , including the energy space (see \cite{MST2,IKT}).} It is furthermotre proven in \cite{deBS2} that any ground state (and in fact any cylindrically symmetric solitary wave) is unstable when $p>\frac{4}{3}.$

 The instability result was improved by Liu \cite{L1} who used invariant sets of the generalized KP  I flow together with the virial argument above to prove the existence of initial data leading to blow-up in finite time of $|u_y(.,t)|_{2}$ when $p\geq \frac{4}{3}$. This leads to  a strong instability result (by finite time blow-up of $|u_y(.,t)|_{2}$ ) of the solitary waves when $2<p<4.$
 
 \medskip

In order to  check to what extent the above results could be extended to the fKP equation one has as a  first step to establish   a fractionary Gagliardo-Nirenberg inequality that allows for the Sobolev embedding of the energy space $Y_\alpha$ into $L^p(\R^2), \, p\leq3.$ 

The following inequality is a special case of    Lemma 2.1  in \cite{BLT} which considers only  $1\leq \alpha\leq 2$ but a close inspection at the proof reveals that it is still valid when $\frac{4}{5}<\alpha <1.$

\begin{lemma}\label{BLT}
Let $\frac{4}{5}<\alpha < 1.$ For any $f\in Y_\alpha$ one has 

$$|f|_3^3\leq c|f|_2^{\frac{5\alpha-4}{\alpha+2}}\|f\|_{H^{\frac\alpha2}_x}^{\frac{18-5\alpha}{2(\alpha+2)}}|\partial_x^{-1}f_y|_2^{\frac12} \, ,$$
where $\|\cdot\|_{H^{\frac{\alpha}2}_x}$ denotes the natural norm on the space $$H_x^{\frac\alpha2}(\R^2)=\lbrace f\in L^2(\R^2) \ : \  D^{\frac\alpha2}_xf \in L^2(\R^2)\rbrace.$$
\end{lemma}

Lemma \ref{BLT} implies obviously the embedding $Y_\alpha \hookrightarrow L^3(\R^2)$ if $\frac{4}{5}<\alpha<1$
and is the starting point for an existence theory of solitary waves to fKPI equations  which will be developed elsewhere \cite{LPS3}. Note that some results for the case $\alpha=1$ (the KPI-Benjamin-Ono equation) are given in \cite{Esf, PS}.
\vspace{0.5cm}

\begin{remark}
Concerning the Cauchy problem for fKPI, one could conjecture a {finite time  blow-up of $|u_y|_0$} when $\frac{4}{5}<\alpha <\frac{4}{3}$ as Liu proved for the gKPI, explaining for instance the (expected) instability of  KPI-BO ground states. 
We refer to a subsequent work \cite{LPS3}  for a study of global {\it weak solutions} to fKP equations.
\end{remark}



\vspace{0.3cm}

\vspace{0.3cm}

 \section{Final remarks}


As already noticed, the precise description of the (expected) instability mechanism  of the solitary waves of \eqref{fKdV} when $\frac{1}{3}<\alpha\leq\frac{1}{2}$ seems out of reach for the moment. According to the numerical simulations in \cite{KS}, the instability seems to be due to blow-up. Recall that this issue is still open for the generalized KdV equation  (that is \eqref{GfKdV} with $\alpha=2$) when $p>4,$ the critical case $p=4$ being treated in \cite {MM}. 

Similarly, the description of the (expected) instability of slow solitary waves of the fBBM equation when $\frac{1}{3}<\alpha\leq\frac{1}{2}$  is not known. Recall that a corresponding rigorous description of solitary waves of the gBBM when $p>4$ is still an open problem.

 On the other hand, the computations in \cite{KS2} seem to indicate that the {\it soliton resolution conjecture} (see \cite{Tao}) is true for both the fKdV and fBBM equations in the stable range $\frac{1}{2}<\alpha\leq 1.$



\begin{merci}
The Authors were partially  supported by the Brazilian-French program in mathematics. J.-C. S. acknowledges support from the project ANR-GEODISP of the Agence Nationale de la Recherche. F.L and D.P. were partially supported by CNPq and FAPERJ/Brazil.
\end{merci}
\bibliographystyle{amsplain}

\end{document}